\documentclass[12pt,a4paper]{amsart}

\usepackage{amsfonts, amsmath, amssymb, amsthm, hyperref}

\usepackage{a4wide}

\newtheorem{thm}{Theorem}

\newtheorem{con}{Conjecture}

\theoremstyle{definition}
\newtheorem{defn}{Definition}

\theoremstyle{remark}

\DeclareMathOperator{\dist}{dist}

\newcommand{\interior}{\mathop{\rm int}}
\renewcommand{\Re}{\mathop{\rm Re}}
\renewcommand{\epsilon}{\varepsilon}

\begin{document}

\title{A note on Makeev's conjectures}

\author{R.N.~Karasev}
\thanks{This research is supported by the Dynasty Foundation, the President's of Russian Federation grant MK-113.2010.1, the Russian Foundation for Basic Research grants 10-01-00096 and 10-01-00139, the Federal Program ``Scientific and scientific-pedagogical staff of innovative Russia'' 2009--2013}

\email{r\_n\_karasev@mail.ru}
\address{
Roman Karasev, Dept. of Mathematics, Moscow Institute of Physics
and Technology, Institutskiy per. 9, Dolgoprudny, Russia 141700}


\begin{abstract}
A counterexample is given for the Knaster-like conjecture of Makeev for functions on $S^2$. Some particular cases of another conjecture of Makeev, on inscribing a quadrangle into a smooth simple closed curve, are solved positively.
\end{abstract}

\maketitle

\section{Introduction}

In~\cite{kna1947} the following conjecture (Knaster's problem) was formulated.

\begin{con}
\label{knaster}
Let $S^{d-1}$ be a unit sphere in $\mathbb R^d$. Suppose we are given $d$ points $x_1, \ldots, x_d\in S^{d-1}$ and a continuous function $f: S^{d-1}\to \mathbb R$. Then there exists a rotation $\rho\in SO(d)$ such that
$$
f(\rho(x_1)) = f(\rho(x_2)) = \dots = f(\rho(x_d)).
$$
\end{con}

This conjecture was shown to be false in~\cite{kasha2003,hiri2005} for certain functions and sets $\{x_i\}$ for large dimensions, namely for $d=61$ and all $d\ge 67$. In this paper we consider some modifications of this problem for functions on $S^2$ ($d=3$). Conjecture~\ref{knaster} was solved positively for $d=3$ in~\cite{flo1955}. In the papers~\cite{dys1951,liv1954,gri1991} it was shown that the similar result holds for $4$ (not $3$) points on $S^2$, when the points form a rectangle. In~\cite{mak1989} this modification of the Knaster problem was solved for $4$ points on $S^2$, when these points are some $4$ vertices of a regular pentagon.

In~\cite{mak1989} (see also \cite[Ch.~I]{mak2003}) it was noted that if $4$ points satisfy the Knaster-like property on $S^2$, they should lie on a single circle (it suffices to consider a linear function $f$), and the following conjecture was formulated.

\begin{con}
\label{makeev-S2}
Let $S^2$ be a unit sphere in $\mathbb R^3$. Suppose we are given $4$ points $x_1, \ldots, x_4\in S^2$, lying on some single circle, and a continuous function $f: S^2\to \mathbb R$. Then there exists a rotation $\rho\in SO(3)$ such that
$$
f(\rho(x_1)) = f(\rho(x_2)) = f(\rho(x_3)) = f(\rho(x_4)).
$$
\end{con}

For some particular classes of functions $f$ this conjecture was proved in~\cite{mak1993,mak2003}. The conjecture is false in general. In Sections~\ref{makeev-S2-ce} and \ref{even-func} we give a counterexample for Conjecture~\ref{makeev-S2}. 

Another conjecture from~\cite{mak1995} is closely related to the functions on a sphere, it could be a consequence of Conjecture~\ref{makeev-S2}, if the latter were true, see~\cite{mak1995,mak2003} for details. 

\begin{con}
\label{makeev-inscr}
Let $C$ be a smooth simple closed curve in $\mathbb R^2$, let $Q$ be some four points on a single circle. Then there is a similarity transform $\sigma$ (with positive determinant), such that $\sigma(Q)\subset C$.
\end{con}

For this conjecture we give some partial solution, formulated as follows.

\begin{thm}
\label{inscr-alt}
Let $C$ be a smooth simple closed curve in $\mathbb R^2$, let $Q=\{a,b,c,d\}$ be some four points on a single circle. Then one of the alternatives folds:

1) There is a similarity transform $\sigma$ (with positive determinant), such that $\sigma(Q)\subset C$;

2) There are two distinct similarity transforms $\sigma_1, \sigma_2$ such that 
$$
\sigma_1(d) = \sigma_2(d),\quad \forall i\ \sigma_i(a),\sigma_i(b),\sigma_i(c)\in C.
$$
\end{thm}

An infinitesimal version of Conjecture~\ref{makeev-inscr} can be proved for convex curves and infinitesimal quadrangles with three coincident vertices. Three coincident vertices give restrictions on the tangent and the curvature of the considered curve.

\begin{defn}
Let $C$ be a $C^2$-smooth curve, and $p\in C$. A circle $\omega$, tangent to $C$ at $p$, and having the curvature, equal to the curvature of $C$ at $p$, is called an \emph{osculating circle} for $C$ at $p$. Note that $\omega$ can be a straight line, if the curvature is zero.
\end{defn}

\begin{thm}
\label{inscr-inf}
Let $C$ be a $C^2$-smooth convex closed curve in $\mathbb R^2$, let $\alpha\in (0, 2\pi)$ be some angle. Then there exist two points $a,b\in C$, such that they lie on the osculating circle $\omega$ for $C$ at $a$, and the counter-clockwise oriented arc $[ab]$ has angular measure $\alpha$.
\end{thm}

The infinitesimal version when two points of the quadrangle coincide (giving a restriction on the tangent) is not established yet, though it seems plausible at least for convex curves.

The author thanks the unknown referee for numerous useful remarks.

\section{The quadrangle and the infinitesimal case in Conjecture~\ref{makeev-S2}}
\label{makeev-S2-ce}

We are going to consider quadrangles $Q(a, b)$ given by the following rule. $Q(a, b)=\{x_1, x_2, x_3, x_4\}$ is on the equator of $S^2$, its points $x_1, x_4$ are opposite, $x_2$ and $x_3$ lie on the different sides of $x_1$, $\dist(x_1, x_2) = a$, and $\dist(x_1, x_3) = b$.

We are going to show that for small enough $a, b$ Conjecture~\ref{makeev-S2} fails for $Q(a, b)$. Assume the contrary and take some $f$ of class $C^\infty$. Then by going to the limit $a, b\to +0$, we use the compactness considerations and obtain the points $x_1\to y_1$, $x_4\to y_4$. 

By going to the limit we have $f(y_1) = f(y_4)$. Since $x_2, x_3\to y_1$, we note that for $f$ in some neighborhood of $y_1$ we can have the following cases.

\begin{enumerate}
\item 
$df(y_1)\not =0$. In this case $y_1$ should be a zero curvature point of the curve (the level line), given by
$$
f(x) = f(y_1),
$$
it follows, that in this case the combination
$$
C(f) = f''_{ss} {f'_t}^2 - 2 f''_{st} f'_s f'_t + f''_{tt} {f'_s}^2
$$
should be zero in $y_1$ for some local coordinates $s, t$, projected from the orthogonal coordinates of the tangent space $TS^2$ at $y_1$;

\item 
$df(y_1) = 0$. In this case the quadratic form, given by the matrix 
$$
\partial^2 f = \left(
\begin{array}{cc}
f''_{ss} & f''_{st}  \\
f''_{st} & f''_{tt}
\end{array}
\right),
$$
cannot be positive-definite, or negative-definite.

\end{enumerate}

We are going to build a counterexample as follows. First we find a $C^\infty$ even function $g$ on $S^2$, and an odd smooth curve $L\subset S^2$ (odd means that, considered as a map $S^1\to S^2$, it is odd) without self-intersections, such that for any point $y\in L$ we have $C(g)(y)\not = 0$ (and therefore $dg(y)\neq 0$ for $y\in L$). Then we consider a smooth odd function $h$, having zeros exactly on $L$ (in our example $L$ is obtained by deforming the equator of the sphere, and $h$ can be obtained from the corresponding deformation of the coordinate function $z$). Put 
$$
f = g + h^3.
$$
For such a function $f$ we note that the points $y_1, y_4$ should be on $L$, since 
$$
f(y_1) - f(y_4) = h^3(y_1) - h^3(y_4) = 2 h^3(y_1),
$$
and 
$$
C(f)(y_1) = C(g)(y_1)\not= 0,
$$ 
since $h$ is cubed and does not affect the first and second derivatives of $f$ and $g$ on $L$. So the infinitesimal case of Conjecture~\ref{makeev-S2} fails for $f$, and the conjecture itself fails for small enough $a, b >0$.

Note that the technique of decomposing $f$ into even and odd parts is used to give some counterexamples to the generalized Knaster conjecture for maps $f :S^n\to \mathbb R^m$ in~\cite{chen1998}.

\section{Construction of the even function}
\label{even-func}

The function $g$ will be constructed as follows. First take $g_0 = Ax^2 + By^2 + Dz^2$, where $A>B>D>0$. Take the line $L_0$ to be the circle $\{(x,y,z)\in S^2 : z=0\}$. It can be easily seen that $C(g_0) < 0$ on $L_0$ except four points $(\pm 1, 0, 0)$ and $(0, \pm 1, 0)$. We can modify $L_0$ in the small neighborhood of $(\pm 1, 0, 0)$ so that the modified line $L_0$ misses $(\pm 1, 0, 0)$. For such modified $L_0$ the inequality $C(g_0) < 0$ holds in this neighborhood, because these points are non-degenerate maximums of $g_0$.

It is a bit more difficult to handle the points $(0, \pm 1, 0)$. In coordinates $(s, t)=(x,z)$ the function $g_0$ up to some affine transformation will have the form
$$
g_0 = s^2 - t^2,
$$
the curve $L_0$ being $\{t = 0\}$. The coordinates $(s,t)$ cannot be used calculate the value $C(g)$ for $(s,t)\neq (0,0)$ in general, but we note the following. If the neighborhood of $(0,0)$ is chosen to be small enough, then the difference between $C(g)$ (in appropriate coordinates) and $g''_{ss} {g'_t}^2 - 2 g''_{st} g'_s g'_t + g''_{tt} {g'_s}^2$ can be made arbitrarily small. Therefore we use the latter expression as $C(g)$ in a neighborhood of $(0,0)$.

We are going to change $g_0$ and $L_0$ simultaneously in some small neighborhood of $(s, t) = (0,0)$. Take some $\varepsilon>0$. Consider 
$$
g(s, t) = s^2 - (t - \phi(s))^2,
$$ 
where the $C^\infty$ function $\phi(s)$ is non-negative, equal to zero for $|s|>2\epsilon$, positive for $|s|<2\epsilon$, strictly convex on $[-2\epsilon,-\epsilon]$ and $[\epsilon, 2\epsilon]$, and strictly concave on $[-\epsilon, \epsilon]$. By the straightforward calculations we find 
$$
\frac18 C(g) = (\phi(s) - t)^2 - (\phi(s) - t)^3\phi''(s) - s^2.
$$
Now consider another $C^\infty$ function $\psi(s)$, equal to zero for $|s| >\epsilon/2$, and positive for $|s| <\epsilon/2$. Let us modify the line $L_0$ so that it becomes the curve 
$$
L = \{(s, t) : t = \phi(s) + \psi(s)\}.
$$ 
On this curve we have
$$
\frac18 C(g)(s) = \psi(s)^2 + \psi(s)^3\phi''(s) - s^2.
$$
On the part of $L$, where $|s| > \epsilon/2$, obviously $C(g) < 0$. On the part of $L$, where $|s| \le \epsilon/2$ and $|\psi(s)|<|s|$, the inequality $C(g) < 0$ is true again. On the part $|\psi(s)|\ge|s|$ it is not true in general, but it becomes true if we multiply $\phi(s)$ by some sufficient large coefficient, leaving $\psi(s)$ the same. If the functions $\phi(s)$ and $\psi(s)$ get too large, we can make a homothety of the whole picture with arbitrarily small factor to make them lesser.

The above construction changes $L_0$ in a small neighborhood of $(0, 0)$ in $(s,t)$ coordinates. The corresponding change of $g_0$ is made for small enough values of $s$ coordinate, but we do not need to extend this change to large values of $t$ coordinate, since the curve $L$ remains in some limited range of $|t|$. Hence, returning to the sphere, we may assume that $g_0$ and $L_0$ are changed in the small neighborhood of $(0, \pm 1, 0)$. It is also clear that everything can be done symmetrically w.r.t the map $(x,y,z)\mapsto(-x,-y,-z)$, so that the resulting function $g$ remains even, and the curve $L$ remains odd. Thus we obtain the required even function on odd curve.

\section{The proofs of the theorems}

\begin{proof}[Proof of Theorem~\ref{inscr-alt}]

First, identify the plane $\mathbb R^2$ with $\mathbb C$, thus the similarity transforms with positive determinant are identified with $\mathbb C$-linear transforms. In the sequel, a ``similarity transform'' means a similarity transform with positive determinant.

Let us choose a smooth parameterization of $C$ by the map $f: \mathbb R\to \mathbb C$ with period $1$. We are going to study the variety of triples $a',b',c'\in C$, such that 
$$
\triangle a'b'c' \sim \triangle abc.
$$
Consider the number $r=\dfrac{c-a}{b-a}\in \mathbb C$. Let the point $a'$ be parameterized by $t\in \mathbb R$, $b'$ by $t+s$, where $s\in (0, 1)$. Consider the corresponding space of pairs of parameters $X=\mathbb R\times I\setminus \mathbb R\times \partial I$, where $I=[0, 1]$ is the standard segment. Now the condition
$$
c' = r(b' - a') + a'\in C
$$
defines a subset $Z\subset X$. From the general position considerations, the curve $C$ can be perturbed (in $C^1$ metric) so that the subset $Z$ becomes a smooth curve in $X$. This can be explained as follows: for a generic (e.g. algebraic) curve $D$ the condition $r(b'-a')+a'\in D$ and its first differential give three independent conditions, therefore for a generic curve $D$ these three conditions cannot hold simultaneously, and therefore $Z$ does not have singularities for a generic $D$. It is easy to see that the statement of the theorem is stable under going to the limit in $C^1$ metric (see also the end of the proof), so the perturbation is allowed.

Let us find the homological intersection of $Z$ with the segment $\{t\}\times I$, let $f(t) = a'$. This intersection is transversal for a generic $t$ and corresponds to a transversal intersection of $C$ with the curve
$$
a' + r(C - a'),
$$
at a point different from $a'$. Since the whole intersection index of two smooth curves is zero, in follows that the intersection  $Z\cap \{t\}\times I$ has index $1$. It follows now that the curve $Z$ must have an unbounded component in $X$, since every bounded component $Z_b$ has index $Z_b\cap \{t\}\times I$ equal to zero. Denote some unbounded component of $Z$ by $Y$. Note that $Z$ is a closed periodic subset of $X$, therefore $Y$ is unbounded with respect to the coordinate $t$, while the parameter $s$ always remains in some segment $[\varepsilon, 1-\varepsilon]$.

Let us show that $Y$ must be periodic w.r.t. the transform $T:(t,s)\mapsto (t+1, s)$. In fact, $T(Y)$ is a connected component of the curve $Z$, the same is true for curves $T^k(Y),\ k\in\mathbb Z$. Since every intersection $Z\cap\{t\}\times I$ is finite, for some $k,l\in\mathbb Z$ the sets $T^k(Y)\cap\{t\}\times I$ and $T^l(Y)\cap\{t\}\times I$ coincide. The set $Y$ is a connected component of $Z$, so we have $T^{k-l}(Y)=Y$. Therefore, the curve $Y$ divides $X$ into two open parts, call them ``top'' $X_+$ and ``bottom'' $X_-$ (w.r.t. $s$). The equality $T(Y)=Y$ follows if the sets $Y$ and $T(Y)$ have nonempty intersection. Assume the contrary: then the curve $T(Y)$ is contained either in $X_+$, or in $X_-$; without loss of generality let it be in $X_-$. Then $T^2(Y)$ is ``under'' $T(Y)$, and therefore in $X_-$. Iterating this reasoning we see that $T^{k-l}(Y)=Y$ is contained in $X_-$. This contradiction proves that $T(Y)=Y$.

Now we parameterize the smooth curve $Y$ by the functions $t(u), s(u)$ so that 
$$
t(u+1) = t(u) + 1,\quad s(u+1) = s(u).
$$

Thus we have parameterized some of the triples $a'(u),b'(u),c'(u)\in C$, similar to $\triangle abc$ so that when the parameter is increased by $1$, the points $a'(u),b'(u),c'(u)$ make a one round turn along $C$, though they may go forth and back along $C$ under this parameterization. Denote $q=\dfrac{d-a}{b-a}$, and 
$$
d'(u) = a'(u) + q(b'(u) - a'(u)).
$$
If the point $d'(u)$ is on $C$, then the first alternative of the theorem holds. Let us find the areas of the curves, parameterized by $a'(u), b'(u), c'(u), d'(u)$, They are given by Green's theorem (up to the factor $i/2$, that is omitted for brevity)
$$
S_a = \int_0^1 a'(u) d\overline{a'(u)},\quad S_b = \int_0^1 b'(u) d\overline{b'(u)},
$$
$$
\quad S_c = \int_0^1 c'(u) d\overline{c'(u)},\quad S_d = \int_0^1 d'(u) d\overline{d'(u)}.
$$
Denote $o$ the circumcenter of the quadrangle $abcd$, $o(u)$ the circumcenter of $a'(u),b'(u),c'(u),d'(u)$. Put 
$$
\alpha(u)=\dfrac{b'(u) - a'(u)}{b - a}, r_a = a - o, r_b = b - o, r_c = c - o, r_d = d - o.
$$
Now we can rewrite the integrals
\begin{multline*}
S_a = \int_0^1 a'(u) d\overline{a'(u)} = \int_0^1 (o(u) + \alpha(u) r_a) d\overline {\left(o(u) + \alpha(u) r_a\right)} = \\
\int_0^1 o(u) d\overline{o(u)} + r_a \int_0^1 \alpha(u) d\overline {o(u)} + \overline{r_a} \int_0^1 o(u) d\overline{\alpha(u)} + r_a\overline{r_a} \int_0^1 \alpha(u)d\overline{\alpha(u)}, 
\end{multline*}
and similar for $S_b, S_b, S_c$ with $r_a$ replaced by $r_b, r_c, r_d$ respectively. Note that the dependence on $\rho=r_a, r_b, r_c, r_d$ has the form
$$
S(\rho) = A + 2 \Re B\rho + D |\rho|^2,
$$
and for $\rho = r_a, r_b, r_c$ (three times with the same $|\rho|^2$) we have $S(\rho) = S_C$, the area of $C$ by Green's theorem. It means that $B = 0$, and the area $S_d$ (in the sence of Green's theorem) equals $S_C$.

If the curve $d'(u)$ has no self-intersections, then its area is indeed $S_C$, and either $d'(u)$ intersects $C$ (in this case the theorem is proved), or the regions bounded by $d'(u)$ and $C$ are disjoint. The latter case is impossible, because it would imply that the vector $d'(u)-a'(u)$ rotates by $0$ when $u$ increases by $1$, but the rotation of $d'(u)-a'(u)$ equals the rotation of $b'(u) - a'(u)$, that is $2\pi$.

If the curve $d'(u)$ has self-intersections, then the second alternative holds. Note that we have perturbed the curve $C$, and when we go to the limit to the original $C$, the self-intersections of $d'(u)$ may become degenerate, i.e. the sizes of loops on the curve $\{d'(u)\}$ tend to zero. But in this case, and the area of these loops should tend to zero, and the above area argument gives points $d'(u)$ that are not lying on $C$, but close enough (the distance depending on the size of loops) to $C$. By going to the limit, $d'$ will be on $C$, and the first alternative of the theorem holds.
\end{proof}

\begin{proof}[Proof of Theorem~\ref{inscr-inf}]

Let $C$ bound a closed region $R$, denote the closure of its complement by $\overline R = \mathbb R^2\setminus \interior R$.

Consider any point $a\in C$, take the osculating circle $\omega(a)$ at $a$, and define $b(a)$ as the point on $\omega$ such that the arc $[ab(a)]$ has angular measure $\alpha$. It is well-known (see~\cite{bose1932} for example), that there are osculating circles that lie entirely in $R$, as well as the osculating circles that lie entirely in $\overline R$. In fact there are at least two circles of every kind (inner and outer). Note that when the point $a$ moves along $C$, the point $b(a)$ moves continuously, sometimes it gets into $R$, and sometimes gets into $\overline R$. Hence for some $a$ (actually, at least four times) $b(a)$ is on $C$.
\end{proof}

\end{document}